 \documentclass{amsart}

\makeindex

\usepackage[all]{xy}
\usepackage{amssymb}

\usepackage{amsmath}

\usepackage{amsthm}

\usepackage{amscd}

\usepackage{amsfonts}

\usepackage{amssymb}

\newtheorem{theorem}{Theorem}[section]
\newtheorem{lemma}[theorem]{Lemma}

\newtheorem{corollary}[theorem]{Corollary}
\theoremstyle{definition}     

\theoremstyle{remark}
\newtheorem{remark}[theorem]{Remark}

\numberwithin{equation}{section}

\begin{document}

\title[Projective surfaces with many nodes]{Projective surfaces with  many nodes}


\author{JongHae Keum}
\address{School of Mathematics, Korea Institute For Advanced Study, Seoul 130-722, Korea}
\email{jhkeum@kias.re.kr}
\thanks{Research supported by Basic Science Research Program through the National Research Foundation(NRF)
of Korea funded by the Ministry of education, Science and
Technology (KRF-2007-C00002)}



\date{August 2009, revised December 2009}


\begin{abstract}
We prove that a smooth projective complex surface $X$, not
necessarily minimal, contains $h^{1,1}(X)-1$ disjoint
$(-2)$-curves if and only if $X$ is isomorphic to the minimal
rational ruled surface ${\bf F}_2$ or ${\bf P}^2$ or a fake
projective plane.

We also describe smooth projective complex surfaces $X$ with
$h^{1,1}(X)-2$ disjoint $(-2)$-curves.
\end{abstract}

\maketitle



\section{Introduction}
Throughout this paper, we work over the field $\mathbb{C}$ of
complex numbers.

A smooth rational curve on a surface with self-intersection $-2$
is called a $(-2)$-curve or a nodal curve as it may be contracted
to give a nodal singularity (conical double point). For a smooth
surface $X$, we denote by $\mu(X)$ the maximum of the cardinality
of a set of disjoint $(-2)$-curves on $X$. Hodge index theorem
implies that
$$\mu(X)\le \rho(X)-1\le h^{1,1}(X)-1,$$ in particular, $X$ contains at most $h^{1,1}(X)-1$ disjoint nodal curves,
where $\rho(X)$ denotes the Picard number and $h^{1,1}(X)$ the
(1,1)-th Hodge number of $X$.

A result of I. Dolgachev, M. Mendes Lopes, and R. Pardini gives a
classification of smooth projective complex surfaces $X$ with
$q(X)=p_g(X)=0$ containing $\rho(X)-1$ disjoint nodal curves
(\cite{DLP}, Theorem 3.3 and Proposition 4.1).

\begin{theorem}\label{dlp}\cite{DLP} Let $X$ be a smooth projective surface,
not necessarily minimal, with $q(X)=p_g(X)=0$. Then
$\mu(X)=h^{1,1}(X)-1$ if and only if $X$ is isomorphic to the
minimal rational ruled surface ${\bf F}_2$  or the complex
projective plane ${\bf P}^2$ or a fake projective plane.
\end{theorem}

Note that $\rho(X)=h^{1,1}(X)$ for a smooth projective surface $X$
with $p_g(X)=0$. The Hirzebruch surface ${\bf F}_2$ contains one
nodal curve, while ${\bf P}^2$ or a fake projective plane contains
none. The latter two cases were not mentioned in \cite{DLP}, as
the authors focused on the case with $\mu(X)>0$.

A $\mathbb{Q}$-homology projective plane is a normal projective
surface having the same $\mathbb{Q}$-homology groups as ${\bf
P}^2$. If a $\mathbb{Q}$-homology projective plane has rational
singularities only, then both the surface and its resolution have
$p_g=q=0$. Theorem 1.1 also gives the following classification of
$\mathbb{Q}$-homology projective planes with nodes only.

\begin{corollary}\label{maincor} Let $S$ be a $\mathbb{Q}$-homology projective plane.
 Assume that all
singularities of $S$ are nodes. Then $S$ is isomorphic to ${\bf
P}^2$ or a fake projective plane or a cone in ${\bf P}^3$ over a
conic curve.
\end{corollary}

In this paper we first show that the condition ``$q(X)=p_g(X)=0$"
in Theorem \ref{dlp} is not necessary.

\begin{theorem}\label{maxnodes} Let $X$ be a smooth projective surface,
not necessarily minimal. Then $\mu(X)=h^{1,1}(X)-1$ if and only if
$X$ is isomorphic to ${\bf F}_2$  or ${\bf P}^2$ or a fake
projective plane.
\end{theorem}

Next, we describe smooth projective complex surfaces $X$ with
$\mu(X)=h^{1,1}(X)-2$.

\begin{theorem}\label{nearmaxnodes} Let $X$ be a smooth projective surface,
not necessarily minimal. Assume that $\mu(X)=h^{1,1}(X)-2$. Then
$X$ belongs to one of the following cases$:$
\begin{enumerate}
\item nef $K_X:$
\begin{enumerate}
\item  a bi-elliptic surface, i.e. a minimal surface of Kodaira
dimension $0$ with $q=1$, $p_g=0$, $h^{1,1}=2;$ \item  a minimal
surface of Kodaira dimension $1$ with $q=1$, $p_g=0$, $h^{1,1}=2;$
\item an Enriques surface with $8$ disjoint nodal curves $;$ \item
a minimal surface of Kodaira dimension $1$ with $q=p_g=0$ whose
elliptic fibration has $2$ reducible fibres of type $I_0^*$ whose
end components give the $8$ disjoint nodal curves $;$ \item a ball
quotient with $q=0$, $p_g=1$, i.e. a minimal surface of general
type with $q=0$, $p_g=1$, $h^{1,1}=2;$ \item a minimal surface of
general type with $q=p_g=0$, $K^2=1,2,4,6,7,8$ containing $8-K^2$
disjoint nodal curves $;$
\end{enumerate}
\item non-nef $K_X:$
\begin{enumerate}
\item  the blowup of a fake projective plane at one point or at
two infinitely near points $;$ \item a relatively minimal
irrational ruled surface; or its blow-up at two infinitely near
points on each of $k\ge 1$ fibres so that each of the $k$ fibres
becomes a string of $3$ rational curves
$(-2)\textrm{---}(-1)\textrm{---}(-2);$ \item a rational ruled
surface ${\bf F}_e$, $e\neq 2;$ \item the blowup of ${\bf F}_e$ at
two infinitely near points on each of $k\ge 1$ fibres so that each
of the $k$ fibres becomes a string of $3$ rational curves
$(-2)\textrm{---}(-1)\textrm{---}(-2);$ \item the blowup of ${\bf
F}_2$ at two infinitely near points away from the negative section
so that one fibre becomes a string of $3$ rational curves
$(-1)\textrm{---}(-2)\textrm{---}(-1);$ \item the blowup of ${\bf
F}_2$ at one point away from the negative section; or equivalently
the blowup of ${\bf F}_1$ at one point on the negative section.
\end{enumerate}
\end{enumerate}
\end{theorem}

We remark that all cases of Theorem \ref{nearmaxnodes} are
supported by an example except the case (1-f) with $K^2=1$, or
$7$.

For the case (1-b), such surfaces can be obtained
 by taking a quotient $(E\times C)/G$ of the
product of an elliptic curve $E$ and a hyperelliptic curve $C$ of
genus $g(C)\ge 2$ by a group $G$ of order 2 acting on $E$ as a
translation by a point of order 2 and on $C$ as the hyperelliptic
involution. (If $g(C)=1$ we get a bi-elliptic surface.)

For (1-c), such Enriques surfaces were completely classified in
\cite{LP}. See also \cite{Kon} and \cite{Keum} for explicit
examples.

For (1-d), the Jacobian fibration of such a surface is a rational
elliptic surface $Y$ with two singular fibres of type $I_0^*$.
(The Jacobian fibration of an elliptic fibration has singular
fibres of the same type as the original fibration (cf.
\cite{CD}).) In other words, such surfaces are torsors of $Y$,
i.e., can be obtained by performing logarithmic transformations on
$Y$. If the orders of logarithmic transformations are $(2,2)$,
then the resulting surface is an Enriques surface belonging to the
case (1-c). Such a rational elliptic surface $Y$ can be
constructed in many ways, e.g., by blowing up the base points of a
specific cubic pencil on ${\bf P}^2$ (\cite{CD}) or by taking a
minimal resolution of a $\mathbb{Z}/2$-quotient of the product of
an elliptic curve $E$ and ${\bf P}^1$ where the group acts on $E$
as the inversion and on ${\bf P}^1$ as an involution. It is easy
to see that any such rational elliptic surface $Y$ is a special
case of (2-d). (Consider a free pencil $|N_1+2C+N_2|$ on $Y$ where
$C$ is a section meeting two simple components $N_1$ and $N_2$ of
the two reducible fibres.)

For the case (1-f) with $K^2=2,4,6$, examples can be found in
\cite{BCGP}. See Remark 4.2.

\bigskip
{\bf Notation}

${\bf F}_e:=$Proj$(\mathcal{O}_{{\bf P}^1}\oplus \mathcal{O}_{{\bf
P}^1}(e))$, $e\ge 0$, a rational ruled surface (Hirzebruch
surface)

${\bf P}^n$: the complex projective $n$-space

$\rho(Y)$: the Picard number of a variety $Y$

$K_Y$: the canonical class of $Y$

$b_i(Y)$: the $i$-th Betti number of $Y$

$e(Y)$: the topological Euler number of $Y$

$c_i(X)$: the $i$-th Chern class of $X$

$e_{orb}(S)$: the orbifold Euler number of a surface $S$ with
quotient singularities only

$h^{i,j}(X)$: the $(i,j)$-th Hodge number of a smooth variety $X$

$q(X):=\dim H^1(X,\mathcal{O}_{X})$ the irregularity of a smooth
surface $X$

$p_g(X):=\dim H^2(X,\mathcal{O}_{X})$ the geometric genus of a
smooth surface $X$

$|G|$: the order of a finite group $G$

$(-m)$-curve: a smooth rational curve on a surface with
self-intersection $-m$

$\mu(Z)$: the maximum of the cardinality of a set of disjoint
$(-2)$-curves on a smooth surface $Z$

\section{The orbifold Bogomolov-Miyaoka-Yau
inequality}

 A singularity $p$ of a normal surface $S$ is called a
quotient singularity if the germ at $p$ is analytically isomorphic
to the germ of $\mathbb{C}^2/G_p$ at the image of the origin $O\in
\mathbb{C}^2$ for some nontrivial finite subgroup $G_p$ of
$GL_2(\mathbb{C})$ not containing quasi-reflections. Brieskorn
classified all such finite subgroups of $GL(2, \mathbb{C})$ in
\cite{Bries}. We call $G_p$ the local fundamental group of the
singularity $p$.

 Let $S$ be a normal projective surface with quotient singularities and $$f : S'
\rightarrow S$$ be a minimal resolution of $S$. For each quotient
singular point $p\in S$, there is a string of smooth rational
curves $E_j$ such that $$f^{-1}(p)=\cup_{j=1}^{l} E_j.$$ It is
well-known that quotient singularities are log-terminal
 singularities. Thus one can compare canonical classes $K_{S'}$ and $K_S$, and write $$K_{S'} \underset{num}{\equiv} f^{*}K_S -
 \sum{\mathcal{D}_p}$$ where $\mathcal{D}_p = \sum_{j=1}^{l}(a_jE_j)$ is an effective $\mathbb{Q}$-divisor supported on $f^{-1}(p)=\cup_{j=1}^{l} E_j$ and $0 \leq a_j < 1$.
It implies that
\[K^2_S = K^2_{S'} - \sum_{p \in Sing(S)}{\mathcal{D}_p^2}.
\]
The coefficients $a_1, \ldots, a_l$ of $\mathcal{D}_p$ are
uniquely determined by the system of linear equations
$$\mathcal{D}_p\cdot E_i =-K_{S'}\cdot E_i = 2+E_i^2\quad (1\le i\le l).$$  In particular, $\mathcal{D}_p=0$ if and only if $p$
is a rational double point.

Also we  recall the orbifold Euler characteristic
$$ e_{orb}(S) := e(S) - \sum_{p \in Sing(S)} \Big ( 1-\frac{1}{|G_p|} \Big )$$
where $e(S)$ is the topological Euler number of $S$, and $|G_p|$
the order of the local fundamental group $G_p$ of $p$.

The following theorem is called the orbifold Bogomolov-Miyaoka-Yau
inequality.

\begin{theorem}[\cite{Sakai}, \cite{Miyaoka}, \cite{KNS}, \cite{Megyesi}]\label{bmy} Let
$S$ be a normal projective surface with quotient singularities
such that $K_S$ is nef. Then
\[
K_{S}^2 \leq 3e_{orb}(S).
\]
In particular,
\[
0 \leq e_{orb}(S).
\]
\end{theorem}

The second (weaker) inequality holds true even if $-K_S$ is nef.

\begin{theorem}[\cite{KM}]\label{bmy2} Let
$S$ be a normal projective surface with quotient singularities
such that $-K_S$ is nef. Then
\[
0\leq e_{orb}(S).
\]
\end{theorem}

The following corollary is well-known (e.g. \cite{HK1}, Corollary
3.4) and immediately follows from Theorems \ref{bmy} and
\ref{bmy2}.

\begin{corollary}\label{bound5}
A $\mathbb{Q}$-homology projective plane with quotient
singularities only has at most $5$ singular points.
\end{corollary}

\section{Proof of Theorem 1.3}

The ``if''-part is trivial.

Let $X$ be a smooth projective surface, not necessarily minimal,
with $h^{1,1}(X)-1$ disjoint nodal curves. We shall show that $X$
is isomorphic to ${\bf F}_2$ or ${\bf P}^2$ or a fake projective
plane. Let
$$f: X\to S$$ be the contraction morphism of the $h^{1,1}(X)-1$
disjoint nodal curves.

Note first that $\rho(S)=1$. Thus $K_S$ is nef or anti-ample.

\medskip
Assume that  $K_S$ is nef.\\
Then we can apply the orbifold Bogomolov-Miyaoka-Yau inequality
(Theorem \ref{bmy}). Note that $K_S^2=K_X^2$. From Noether formula
$$K_X^2=12\{1-q(X)+p_g(X)\}-\{2-4q(X)+h^{1,1}(X)+2p_g(X)\}.$$
Also we have
$$e_{orb}(S)=e(S)-\frac{h^{1,1}(X)-1}{2}.$$
Since $e(S)=e(X)-(h^{1,1}(X)-1)=3-4q(X)+2p_g(X)$, Theorem
\ref{bmy} implies that
$$12(1-q+p_g)-(2-4q+h^{1,1}+2p_g)\le 3\Big(3-4q+2p_g-\frac{h^{1,1}-1}{2}\Big),$$
i.e., $$4q+4p_g+\frac{h^{1,1}}{2}\le \frac{1}{2},$$ hence
$$q(X)=p_g(X)=0, \quad h^{1,1}(X)=1.$$
In particular,  $b_2(X)= 1$, $e(X)=3$, $K_X^2=9$.  Note that
$K_X=f^*K_S$ is nef, hence $X$ is not rational. Thus, by
classification theory of complex surfaces (see \cite {BHPV}), $X$
must be a fake projective plane, i.e. a smooth surface of general
type with $q=p_g=0$, $K^2=9$.

\medskip
Assume that  $-K_S$ is ample.\\ Then  $-K_X=-f^*K_S$ is nef and
non-zero, hence $X$ has Kodaira dimension $\kappa(X)=-\infty$.
Suppose $X$ is not rational. Then there is a morphism $g:X\to C$
onto a curve of genus $\ge 1$, with general fibres isomorphic to
${\bf P}^1$. Since a curve of genus $\ge 1$ cannot be covered by a
rational curve, we see that all nodal curves of $X$ are contained
in a union of fibres. This implies that $S$ has Picard number $\ge
2$, a contradiction. Thus $X$ is rational. Now by Theorem 3.3 of
\cite{DLP},
 $X\cong {\bf F}_2$ or ${\bf P}^2$.\\ Here we give an alternative
 proof. Since we assume that $X$ is rational, $S$ is a $\mathbb{Q}$-homology projective
 plane with nodes only. Let $k$ be the number of nodes on $S$.
 Then $k\le 5$ by Corollary \ref{bound5}.
 Note that $b_2(X)=1+k$, so $K_X^2=9-k$. Let $L$ be the sublattice
 of the cohomology lattice of $X$ generated by the class of $K_X$ and the classes of the $k$ nodal
 curves. Then $L$ is of finite index in the cohomology lattice that is unimodular, hence $|\det(L)|$ is a square integer.
 Note that $|\det(L)|=(9-k)2^k$. If $k\le 5$, then it is a square integer only if $k=0$ or 1.
 If $k=0$, then $X\cong {\bf P}^2$. If $k=1$, then $K_X^2=8$ and $\rho(X)=2$, hence $X\cong {\bf
 F}_2$.\\
This completes the proof of Theorem 1.2.

\begin{remark} Proposition 4.1 of
\cite{DLP} was also proved by using the orbifold
Bogomolov-Miyaoka-Yau inequality. Our proof is just a slight
refinement of their argument.
\end{remark}

\section{Proof of Theorem \ref{nearmaxnodes}}

For a smooth surface $Z$, we denote by $\mu(Z)$ the maximum of the
cardinality of a set of disjoint $(-2)$-curves of $Z$. The
following useful lemma is due to M. Mendes Lopes and R. Pardini.

\begin{lemma}\label{lp} Let X be a smooth surface with Kodaira dimension $\kappa(X)\ge 0$. Let $\phi:X\to Y$ be the map to
the minimal model, and let $r:=\rho(X)-\rho(Y)$. Then
$$\mu(X)\le\mu(Y)+\frac{r}{2}.$$
\end{lemma}

\begin{proof} The proof is essentially contained in the proof of Proposition 4.1 of
\cite{DLP}.

Use induction on $r$.

When $r=0$, it is trivial.

Assume $r>0$. Write
$$K_X=\phi^*K_Y+E$$ and let $C_1, \ldots, C_{\mu(X)}$ be disjoint $(-2)$-curves on $X$.
For each $i$ there are 2 possibilities:
\begin{enumerate}
\item $C_i$ is exceptional
for $f$, hence $(\phi^*K_Y)C_i=0$. \item $C_i$ is not exceptional
for $\phi$. Then since $K_XC_i=0$ and $K_Y$ is nef, we see that
$(\phi^*K_Y)C_i=0$, $EC_i=0$, hence $C_i$ is disjoint from the
support of $E$. \end{enumerate}

Let $E_1$ be an irreducible $(-1)$-curve of $X$ and let $X_1$ be
the surface obtained by blowing down $E_1$. If $E_1$ does not
intersect any of the $C_i$'s, then the $C_i$'s give $\mu(X)$
disjoint $(-2)$-curves on $X_1$, hence $\mu(X)\le\mu(X_1)$ and the
statement follows by induction, i.e.,
$$\mu(X)\le \mu(X_1)\le\mu(Y)+\frac{r-1}{2}.$$

So assume that $E_1C_1>0$. By the above remark, this implies that
$C_1$ is exceptional for $\phi$. In particular, we have
$C_1E_1=1$. Notice that $E_1C_i=0$ for every $i>1$. Indeed, if,
say, $E_1C_2=1$, then the images of $C_1$ and $C_2$ in $X_1$ are
$(-1)$-curves that intersect, contradicting the assumption that
$\kappa(X_1)\ge 0$. Hence the image of $C_1$ in $X_1$ is a
$(-1)$-curve that can be contracted to get a surface $X_2$ with
$\mu(X)-1$ disjoint $(-2)$-curves, and again we get the result by
induction, i.e.,
$$\mu(X)-1\le \mu(X_2)\le\mu(Y)+\frac{r-2}{2}.$$
\end{proof}

Now we prove Theorem \ref{nearmaxnodes}. Let
$$f: X\to S$$ be the contraction morphism of the
disjoint nodal curves $C_1, \ldots, C_{\mu(X)}$, where
$\mu(X)=h^{1,1}(X)-2$.

Note first that $K_X=f^*K_S$. Thus $K_X$ is nef if and only if
$K_S$ is nef.

\medskip
Assume that  $K_X$ is nef.\\
Then we again apply the orbifold Bogomolov-Miyaoka-Yau inequality
(Theorem \ref{bmy}). In this case we have
$$K_S^2=K_X^2=12\{1-q(X)+p_g(X)\}-\{2-4q(X)+h^{1,1}(X)+2p_g(X)\},$$
$$e_{orb}(S)=e(S)-\frac{h^{1,1}(X)-2}{2}=4-4q(X)+2p_g(X)-\frac{h^{1,1}(X)-2}{2}.$$ Thus Theorem
\ref{bmy} implies that
$$12(1-q+p_g)-(2-4q+h^{1,1}+2p_g)\le 3\Big(4-4q+2p_g-\frac{h^{1,1}-2}{2}\Big),$$
and hence, $$4q+4p_g+\frac{h^{1,1}}{2}\le 5.$$ This inequality has
the following solutions:

\begin{itemize} \item (i) $q(X)=1$, $p_g(X)=0$, $h^{1,1}(X)=2$;
\item (ii) $q(X)=0$, $p_g(X)=1$, $h^{1,1}(X)=2$; \item (iii)
$q(X)=p_g(X)=0$, $2\le h^{1,1}(X)\le 10$.
\end{itemize}

Assume the case (i). In this case, $e(X)=0$ and $K_X^2=0$. Since
$X$ is a minimal surface of Kodaira dimension $\kappa(X)\ge 0$,
the classification theory of complex surfaces (cf. \cite{BHPV})
shows that $X$ belongs to the case (1-a) or (1-b).

Assume the case (ii). In this case, $e(X)=6$ and $K_X^2=18$.
Hence, $X$ is of general type. Since $3e(X)=K_X^2$, it is a ball
quotient. This gives the case (1-e).

Assume the case (iii). In this case, $e(X)=h^{1,1}(X)+2$ and
$K_X^2=10-h^{1,1}(X)$.

 If $h^{1,1}(X)=10$, then $e(X)= 12$ and
$K_X^2=0$. Since $X$ is a minimal surface of Kodaira dimension
$\kappa(X)\ge 0$, $X$ is an Enriques surface or a minimal surface
of Kodaira dimension $\kappa(X)=1$. This gives the case (1-c) and
(1-d). In the latter case, a fibre of the elliptic fibration of
$X$ is a rational multiple of $K_X$, hence the 8 nodal curves must
be contained in fibres of the elliptic fibration. By the formula
for computing the topological Euler number of a fibration (cf.
\cite{BHPV}, Chap. III) this is possible only if the reducible
fibres are two fibres of type $I_0^*$ and the eight nodal curves
are the end-components of these fibres.

 If $2\le h^{1,1}(X)\le 9$, then
$8\ge K_X^2\ge 1$. Since $X$ is a minimal surface of Kodaira
dimension $\kappa(X)\ge 0$, $X$ is of general type. By Theorem
\ref{maxnodes}, any minimal surface $X$ of general type with
$p_g(X)=0$ and $K_X^2=8$ cannot contain a $(-2)$-curve, hence
belongs to the case (1-f).

We claim that $K_X^2\neq 3, 5$. This can be proved by a lattice
theoretic argument. Let $L$ be the cohomology lattice
$H^2(X,\mathbb{Z})/$(torsion), which is an odd unimodular lattice
of signature $(1, h^{1,1}(X)-1)$. Let $M$ be the sublattice of $L$
generated by the classes of the nodal curves $C_1, \ldots,
C_{\mu(X)}$ where $\mu(X)=h^{1,1}(X)-2$.
 Consider the homomorphism of quadratic forms of finite abelian groups
 $$\tau: M/2M\to L/2L.$$
 Note that $M/2M\cong (\mathbb{Z}/2\mathbb{Z})^{\mu(X)}$ is totally
 isotropic, and $L/2L\cong (\mathbb{Z}/2\mathbb{Z})^{\mu(X)+2}$.
 Assume that $K_X^2=3$. Then $\mu(X)=5$, so the kernel $\ker(\tau)$
 must have length $\ge 2$. If
 $\sum_{j=1}^{k}C_{i_j}$(mod $2M$) $\in\ker(\tau)$, then
 $\sum_{j=1}^{k}C_{i_j}=2D+$torsion for some divisor $D$. Since $D\cdot K_X=0$, $D^2$ is an
 even integer. This implies that $k$ is a multiple of 4. This
 means that any non-trivial element of $\ker(\tau)$ is a sum of 4
 members of $C_1, \ldots, C_{5}$. But this is impossible since $\ker(\tau)$
 has length $\ge 2$. Assume that $K_X^2=5$. Then $\mu(X)=3$, so the kernel $\ker(\tau)$
 must have length $\ge 1$. But no linear combination of $C_1, C_2, C_3$ gives a non-trivial element of $\ker(\tau)$.
This gives the case (1-f).

\medskip
Assume that  $K_X$ is not nef and $\kappa(X)\ge 0$.\\ In this case
$X$ is not minimal. Consider the map $\phi:X\to Y$ to the minimal
model, and let $$r:=\rho(X)-\rho(Y).$$ By Lemma \ref{lp},
$$h^{1,1}(Y)+r-2=h^{1,1}(X)-2=\mu(X)\le\mu(Y)+\frac{r}{2}\le
h^{1,1}(Y)-1+\frac{r}{2},$$ hence $$r\le 2.$$ If $r=1$, then the
above inequality shows that
$$h^{1,1}(Y)-1=\mu(X)\le\mu(Y)+\frac{1}{2},$$ hence $\mu(Y)=
h^{1,1}(Y)-1$. So by Theorem \ref{maxnodes} $Y$ is a fake
projective plane and $\mu(X)=h^{1,1}(Y)-1=0$.\\ If $r=2$, then the
above inequality shows that
$$h^{1,1}(Y)=\mu(X)\le\mu(Y)+1\le
h^{1,1}(Y),$$ hence $\mu(Y)= h^{1,1}(Y)-1$. So by Theorem
\ref{maxnodes} $Y$ is a fake projective plane and
$\mu(X)=h^{1,1}(Y)=1$. This gives the case (2-a).

\medskip
Assume that  $\kappa(X)=-\infty$ and $X$ is irrational.\\ In this
case $X$ is an irrational ruled surface. If $X$ is relatively
minimal, then $\mu(X)=h^{1,1}(X)-2=0$. Assume that
$\mu(X)=h^{1,1}(X)-2>0$. The $(-2)$-curves must be contained in
the union of fibers of the Albanese fibration $a_X$ on $X$. Let
$\phi:X\to Y$ be the map to a relatively minimal irrational ruled
surface. Then
$$\mu(X)=h^{1,1}(X)-2=h^{1,1}(Y)+\rho(X)-\rho(Y)-2=\rho(X)-\rho(Y),$$
i.e. the number of disjoint nodal curves on $X$ is the same as the
the number of blowups from $Y$ to $X$. This is possible only if
the number of nodal curves contained in each reducible fibre of
$a_X$ is the same as the number of blowups on the corresponding
fibre of $Y$. The only possibility is that each reducible fibre of
$a_X$ is a string of three smooth rational curves
$(-2)\textrm{---}(-1)\textrm{---}(-2)$ obtained by blowing up
twice. This gives the case (2-b).

\medskip
Assume that  $\kappa(X)=-\infty$ and $X$ is rational.\\ This case
has been classified in \cite{DLP}, Theorem 3.3 and Remark 3. This
gives the cases (2-c), (2-d), (2-e), (2-f).

This completes the proof of Theorem \ref{nearmaxnodes}.

\begin{remark} (1) There are examples of the case (1-f) with $K^2=6, 4, 2$, as given in \cite{BCGP}.
In the paper they give a complete classification of the surfaces
$Y$ occurring as the minimal resolution of a surface
$Z:=(C_1\times C_2)/G$, where $G$ is a finite group with an
unmixed action on a product of smooth projective curves $C_1\times
C_2$ of respective genera $\ge 2$, and such that (i) $Z$ has only
rational double points as singularities, (ii) $q(Y)=p_g(Y)=0$. In
particular they show that $Z$ has only nodes as singularities, and
the number of nodes is even and equal to $t:=8-K_Z^2$ (see
Corollary 5.3, ibid). Furthermore, they give examples with
$t=2,4,6$. The case $t = 0$, i.e., $G$ acts freely on $C_1\times
C_2$, was completely classified in \cite{BCG}.

The cases $t=6,4$ can also be obtained as the quotient of a
minimal surface of general type with $K^2=8$ and $p_g=0$ by an
action of $(\mathbb{Z}/2\mathbb{Z})^{2}$, or by an action of
$\mathbb{Z}/2\mathbb{Z}$, where each non-trivial involution has
isolated fixed points only. This was confirmed by Ingrid Bauer.

(2) We do not know the existence of the case (1-f) with $K^2=1$,
i.e. a Godeaux surface with 7 disjoint nodal curves. However,
there is a possible construction of such an example. If one can
find a minimal surface of general type with $K^2=8$ and $p_g=0$
admitting an action of $(\mathbb{Z}/2\mathbb{Z})^{3}$, each of the
7 involutions having isolated fixed points only, then the quotient
has the minimal resolution with $K^2=1$, $p_g=0$, and 7 disjoint
nodal curves.

(3) We do not know the existence of the case (1-f) with $K^2=7$.
\end{remark}

\begin{remark}
The case (2-a) gives counterexamples to Proposition 4.1 of
\cite{DLP}. Indeed the authors, though their proof was correct,
overlooked the case of fake projective planes for the minimal
case, and consequently the case of blowups of fake projective
planes for the non-minimal case as they used induction on the
number of blowups from the minimal model. Thus the first statement
of their proposition holds true except for the case where Y is a
fake projective plane, and the second statement except for the
case where Y is the blowup of a fake projective plane at one point
or at two infinitely near points.
\end{remark}

\bigskip
{\bf Acknowledgements.} I thank Margarida Mendes Lopes and Rita
Pardini for helpful conversations and for allowing me to use their
result, Lemma \ref{lp}, which plays a key role in the proof of
\ref{nearmaxnodes}. I also thank Ingrid Bauer for informing me of
the result \cite{BCGP}. Finally I am grateful to the referee for
careful reading and for pointing out one missing case in the
statement of Theorem 1.4.

\bibliographystyle{amsplain}

\end{document}